\documentclass[12pt, reqno]{amsart}
\usepackage{graphicx}
\usepackage{hyperref}
\usepackage[caption = false]{subfig}
\usepackage{amsthm}
\usepackage{amssymb}
\usepackage{mathrsfs}
\usepackage{mathtools}
\usepackage{enumerate}
\usepackage{amssymb}
\usepackage{wrapfig}
\usepackage{float}
\allowdisplaybreaks

\usepackage[twoside,paperwidth=190mm, paperheight=297mm, top=35mm, bottom=35mm, left=30mm, right=26mm, marginparsep=3mm, marginparwidth=40mm]{geometry}

\numberwithin{equation}{section}

\theoremstyle{plain}

\newtheorem{theorem}{Theorem}[section]
\newtheorem{corollary}[theorem]{Corollary}

\newtheorem{lemma}{Lemma}[section]

\theoremstyle{definition}
\newtheorem{definition}[theorem]{Definition}
\theoremstyle{remark}
\newtheorem{remark}{Remark}[section]

\makeatother

\begin{document}

\title{Bohr-Rogosinski phenomenon for $\mathcal{S}^*(\psi)$ and $\mathcal{C}(\psi)$}
	\thanks{K. Gangania thanks to University Grant Commission, New-Delhi, India for providing Junior Research Fellowship under UGC-Ref. No.:1051/(CSIR-UGC NET JUNE 2017).}

	\author[Kamaljeet]{Kamaljeet Gangania}
	\address{Department of Applied Mathematics, Delhi Technological University,
		Delhi--110042, India}
	\email{gangania.m1991@gmail.com}
	
	\author[S. Sivaprasad Kumar]{S. Sivaprasad Kumar}
	\address{Department of Applied Mathematics, Delhi Technological University,
		Delhi--110042, India}
	\email{spkumar@dce.ac.in}

\maketitle	
	
\begin{abstract} 
	  In Geometric function theory, occasionally attempts have been made to solve a particular problem for the Ma-Minda classes, $\mathcal{S}^*(\psi)$ and $\mathcal{C}(\psi)$ of univalent starlike and convex functions, respectively. Recently, a popular radius problem generally known as Bohr's phenomenon has been studied in various settings, however a little is know about Rogosinski radius. In this article, for a fixed $f\in \mathcal{S}^*(\psi)$ or $\mathcal{C}(\psi),$ the class of analytic subordinants $S_{f}(\psi):= \{g : g\prec f  \} $ is studied for the Bohr-Rogosinski phenomenon in a general setting. It's applications to the classes $\mathcal{S}^*(\psi)$ and $\mathcal{C}(\psi)$ are also shown.
\end{abstract}
\vspace{0.5cm}
	\noindent \textit{2010 AMS Subject Classification}. Primary 30C45, 30C50, Secondary 30C80.\\
	\noindent \textit{Keywords and Phrases}. Subordination, Radius problem, Bohr Radius, Bohr-Rogosinski radius.

\maketitle
	
\section{Introduction}
\label{intro}
Let $\mathcal{A}$ denote the class of analytic functions of the form $f(z)=z+\sum_{k=2}^{\infty}a_kz^k$ in the open unit disk $\mathbb{D}:=\{z: |z|<1\}$. Using subordination~\cite{subbook}, Ma and Minda \cite{minda94} (also see \cite{ganga-iranian}) introduced the unified class of univalent starlike and convex functions defined as follows:
\begin{equation*}
\mathcal{S}^*(\psi):= \biggl\{f\in \mathcal{A} : \frac{zf'(z)}{f(z)} \prec \psi(z) \biggl\}
\end{equation*}
and
\begin{equation*}\label{mindaclass}
\mathcal{C}(\psi):= \biggl\{f\in \mathcal{A} : 1+\frac{zf''(z)}{f'(z)} \prec \psi(z) \biggl\},
\end{equation*}
where  $\psi$ is analytic and univalent with $\Re{\psi(z)}>0$, $\psi'(0)>0$, $\psi(0)=1$ and $\psi(\mathbb{D})$ is symmetric about real axis. Note that $\psi \in \mathcal{P}$, the class of normalized Carath\'{e}odory functions. Also when $\psi(z)=(1+z)/(1-z)$, $\mathcal{S}^*(\psi)$  and $\mathcal{C}(\psi)$ reduces to the standard classes $\mathcal{S}^*$ and $\mathcal{C}$ of univalent starlike and convex functions. 

In GFT, radius problems have a rich history which is being followed till today, see the recent articles \cite{kumar-2019,ganga_cmft2021,janow,kumarravi2016,Kumar-cardioid,ganga-iranian,kumarG-2020,naveen14}. In 1914, Harald Bohr \cite{bohr1914} proved the following remarkable radius problem related to the power series:
\begin{theorem}[Bohr's Theorem, \cite{bohr1914}]\label{BohrTheorem}
	Let $g(z)=\sum_{k=0}^{\infty}a_kz^k$ be an analytic function in $\mathbb{D}$ and $|g(z)|<1$ for all $z\in \mathbb{D}$, then
	\begin{equation*}
	\sum_{k=0}^{\infty}|a_k||z|^k\leq1, \quad \text{for} \quad |z|\leq\frac{1}{3}.
	\end{equation*}
\end{theorem}

Bohr actually proved the above result for $r\leq1/6$. Further Wiener, Riesz and Shur independently sharpened the result for $r\leq1/3$. Presently, the Bohr inequality for functions mapping unit disk onto different domains, other than unit disk is an active area of research. For the recent development on Bohr-phenomenon, see the articles \cite{ali2017,jain2019,bhowmik2018,boas1997,muhanna10,muhanna11,muhanna14} and references therein.
The concept of Bohr phenomenon in terms of subordination can be described as:
\begin{definition}[Muhanna, \cite{muhanna10}]
	Let $f(z)=\sum_{k=0}^{\infty}a_kz^k$ and $g(z)=\sum_{k=0}^{\infty}b_kz^k$ are analytic in $\mathbb{D}$ and $f(\mathbb{D})=\Omega$. For a fixed $f$, consider a class of analytic functions $S(f):=\{g : g\prec f\}$ or equivalently $S(\Omega):=\{g : g(z)\in \Omega\}$. Then the class $S(f)$ is said to satisfy Bohr-phenomenon, if there exists a constant $r_0\in (0,1]$ satisfying the inequality $
	\sum_{k=1}^{\infty}|b_k|r^k \leq d(f(0),\partial\Omega)$
	for all $|z|=r\leq r_0$ and $g \in S(f)$, where $d(f(0),\partial\Omega)$ denotes the Euclidean distance between $f(0)$ and the boundary of $\Omega=f(\mathbb{D})$. The largest such $r_0$ is called the Bohr-radius.
\end{definition}

In 2014, Muhanna et al. \cite{muhanna14} proved  the Bohr phenomenon for $ S(W_{\alpha})$, where  $W_{\alpha}:=\{w\in\mathbb{C} : |\arg{w}|<\alpha\pi/2, 1\leq \alpha \leq2\},$  which  is a Concave-wedge domain (or exterior of a compact convex set) and the class $R(\alpha,\beta, h)$ defined by $R(\alpha,\beta,h) := \{f\in \mathcal{A} : f(z):=g(z)+\alpha z g'(z)+\beta z^2 g''(z)\prec h(z), g\in \mathcal{A}\},$ where $h$ is a convex function (or starlike) and $R(\alpha,\beta, h)\subset S(h)$.	In 2018, Bhowmik and Das \cite{bhowmik2018} proved the Bohr-phenomenon for the classes:
$S(f)=\{g\in\mathcal{A} : g\prec f\; \text{and}\; f\in \mu(\lambda) \}$, where  $\mu(\lambda)=\{f\in\mathcal{A} : |(z/f(z))^2f'(z)-1|<\lambda, 0<\lambda\leq1\}$ and $S(f)=\{g\in \mathcal{A}: g\prec f \;\text{and}\; f\in \mathcal{S}^*(\alpha), 0\leq\alpha\leq1/2\}$,
where $\mathcal{S}^*(\alpha)$ is the well-known class of starlike functions of order $\alpha$.

In the aforesaid work, the role of the sharp coefficient's bound of $f$ was prominent to achieve the respective Bohr radius for the class $S(f)$, see~\cite{jain2019,ganga-iranian,gangania-bohr}. But in general, the sharp coefficient's bounds for functions in a given class are not available, for example see~\cite{kumar-2019,kumarravi2016,Kumar-cardioid,ganga-iranian,naveen14}, thus certain power series inequalities are needed. In this direction, Bhowmik and das obtained the following important inequality to achieve the Bohr radius for the class $S(f)$, where $f\in \mu(\lambda)$ and $\mathcal{S}^*(\alpha), 0\leq\alpha\leq1/2\}$ respectively:
\begin{lemma}[\cite{bhowmik2018}]\label{Bhowmik-lemma}
	let $f(z)=\sum_{n=0}^{\infty}a_n z^n$ and $g(z)=\sum_{k=0}^{\infty}b_k z^k$ be analytic in $\mathbb{D}$ and $g\prec f$. Then
	$$\sum_{k=0}^{\infty}|b_k|r^k \leq \sum_{n=0}^{\infty}|a_n|r^n, \quad \text{for}\quad |z|\leq\frac{1}{3}. $$ 
\end{lemma}	

Motivated by the class $S(f)$, Kumar and Gangania in~\cite[Sec.~5]{gangania-bohr} further used the above Lemma~\ref{Bhowmik-lemma} in the absence of the sharp coefficient's bounds of $f$ to study the Bohr phenomenon for the class $S_{f}(\psi)$, which eventually holds for the class $\mathcal{S}^*(\psi)$:
\begin{definition}
	Let $f\in \mathcal{S}^*(\psi)$ or $\mathcal{C}(\psi)$ be fixed. Then the class of subordinants functions $g$ is defined as:
	\begin{equation*}\label{bohrclass}
	S_{f}(\psi):= \biggl\{g(z)=\sum_{k=1}^{\infty}b_k z^k : g \prec f \biggl \}.
	\end{equation*}
\end{definition}
\begin{theorem}\cite[Theorem~5.1]{gangania-bohr}
	Let $r_{*}$ be the Koebe-radius for the class $\mathcal{S}^*(\psi),$  $f_0(z)$ be given by the equation~\eqref{int-rep} and $g(z)=\sum_{k=1}^{\infty}b_k z^k \in S_{f}(\psi)$. Assume  $f_0(z)=z+\sum_{n=2}^{\infty}t_n z^n$ and $\hat{f}_0(r)=r+\sum_{n=2}^{\infty}|t_n|r^n$.
	Then   $S_{f}(\psi)$ satisfies the Bohr-phenomenon
	\begin{equation*}
	\sum_{k=1}^{\infty}|b_k|r^k \leq d(f(0),\partial\Omega),\quad \text{for}\; |z|=r\leq r_b,
	\end{equation*}
	where $r_b=\min\{r_0, 1/3 \}$, $\Omega=f(\mathbb{D})$ and $r_0$ is the least positive root of the equation
	$$\hat{f}_0(r)=r_{*}.$$
	The result is sharp when $r_b=r_0$ and $t_n>0$.
\end{theorem}

Note that Muhanna et al.~\cite{Muhanna-2021} recently discussed the Bohr type of inequalities for the $k$-th section for the analytic functions $f(z)=\sum_{n=0}^{\infty}a_n z^n$ using the Bohr Operator
\begin{equation*}
M_r(f)= \sum_{n=0}^{\infty}|a_n||z^n|=  \sum_{n=0}^{\infty}|a_n| r^n.
\end{equation*}
Paulsen and Singh~\cite{paulsen-Singh} using this operator provided an simple elementary proof of the Bohr's Theorem~\ref{BohrTheorem} and extended it to the Banach algebras (for the basic important discussion, see \cite{Muhanna-2021,paulsen-Singh}). Now for the simplicity and further discussion, we define the following basic operator for $f$, where $S^N(f(z))=\sum_{n=N}^{\infty}a_n z^n$:
\begin{equation*}
M^{N}_r(f)= \sum_{n=N}^{\infty}|a_n||z^n|=  \sum_{n=N}^{\infty}|a_n| r^n,
\end{equation*}
and thus the following observations hold for $|z|=r$ for each $z\in \mathbb{D}$
\begin{enumerate}[$(i)$]
	\item $M^{N}_r(f)\geq0$, and $M^{N}_r(f)=0$ if and only if $f\equiv0$
	\item $M^{N}_r(f+g)\leq M^{N}_r(f)+M^{N}_r(g)$
	\item $M^{N}_r(\alpha f) = |\alpha| M^{N}_r(f)$ for $\alpha\in\mathbb{C}$
	\item $M^{N}_r(f.g) \leq M^{N}_r(f). M^{N}_r(g)$
	\item $M^{N}_r(1)=1.$
\end{enumerate}
Using this operertor, we now can get similar type of results as obtained by Muhanna et al.~\cite{Muhanna-2021} for the interim $k$-th sections $S^{N}_{k}(f(z))=\sum_{n=N}^{k}a_n z^n$ and the function $S^N(f(z))$.

In analogy with Bohr's Theorem, there is also the notion of Rogosinski radius, however a little is known about Rogosinski radius as compared to Bohr radius, which is defined as follows, also see \cite{Landau-1986,Rogosinski-1923,Schur-1925}:
\begin{theorem}[Rogosinski Theorem]
	If $g(z)=\sum_{k=0}^{\infty}b_k$ with $|f(z)|<1$, then for every $N\geq1$ we have
	\begin{equation*}
	\left| \sum_{k=0}^{N-1}b_k z^k \right| \leq1, \quad\text{for}\quad |z|\leq \frac{1}{2}.
	\end{equation*}
	The radius $1/2$ is called the Rogosinski radius.
\end{theorem}

Kayumov et al.~\cite{Kayumov-2021} considered a new quantity, called Bohr-Rogosinski sum, which is described as follows:
\begin{equation*}
|g(z)|+ \sum_{k=N}^{\infty}|b_k||z|^k, \quad |z|=r.
\end{equation*}
For the case $N=1$, note that this sum is similar to the Bohr's sum, where $g(0)$ is replaced by $|g(z)|$. We also refer the readers to see~\cite{Aizenberg-2012,Alkha-2020}. Now we say the family $S(f)$ has Bohr-Rogosinski phenomenon, if there exists $r^{f}_{N} \in (0,1]$ such that the inequality: 
$$|g(z)|+ \sum_{k=N}^{\infty}|b_k||z|^k \leq |f(0)|+ d(f(0),\partial \Omega)$$
holds for $|z|=r \leq r^{f}_{N}$. The largest such $r^{f}_{N}$ is called the Bohr-Rogosinski radius. Authors~\cite{Kayumov-2021} also proved the following interesting results:
\begin{theorem}\cite[Theorem~5-6]{Kayumov-2021}
	Let $g\in S(f)$, where $f$ is univalent in $\mathbb{D}$. Then for each $m,N \in \mathbb{N}$, the inequality
	\begin{equation*}
	|g(z^m)|+ \sum_{k=N}^{\infty}|b_k||z|^k \leq |f(0)|+ d(f(0), \partial \Omega)
	\end{equation*}
	holds for $|z|=r\leq r^{f}_{m,N}$, where $r^{f}_{m,N}$ is the smallest positive root of:
	\begin{equation*}
	4r^m- (1-r^m)^2+ 4r^N (N(1-r)+r) \left( \frac{1-r^m}{1-r} \right)^2 =0.
	\end{equation*}
	The radius is sharp for the Koebe function $z/(1-z)^2$. Moreover, if $f$ is convex (univalent) in $\mathbb{D}$, then $r^{f}_{m,N}$ is the smallest positive root of:
	\begin{equation*}
	3r^m- 1+ 2r^N \left( \frac{1-r^m}{1-r} \right) =0.
	\end{equation*}
	The radius is sharp for the convex function $z/(1-z)$.
\end{theorem}

Motivated by the above work, let us now introduce the Bohr-Rogosinski phenomenon for the class of analytic subordinants $S_{f}(\psi)$:
\begin{definition}\label{def-Bohr-Rogo}
	The class $S_{f}(\psi)$ has a Bohr-Rogosinski phenomenon, if there exists an $0<r_0 \leq1$ such that
	\begin{equation*}
	|g(z)|+\sum_{k=N}^{\infty}|b_k||z|^k \leq d(f(0), \partial \Omega)
	\end{equation*}
	for $|z|=r\leq r_0$, where $N\in \mathbb{N}$, $\Omega=f(\mathbb{D})$ and  $d(f(0),\partial\Omega)$ denotes the Euclidean distance between $f(0)$ and the boundary of $\Omega$.
\end{definition}	

Note that $\mathcal{S}^*(\psi)\subset \bigcup_{f\in \mathcal{S}^*(\psi)} S_{f}(\psi)$. Further, the connection between the Bohr-Rogosinski and Bohr phenomenon can be seen through Definition~\ref{def-Bohr-Rogo}, if we replace $|g(z)|$ by $|g(z^m)|$, where $m\in \mathbb{N}$, and then consider the special case by taking $m\rightarrow \infty$ with $N=1$.  In Section~\ref{sec-1}, for a fixed $f\in \mathcal{S}^*(\psi)$ or $\mathcal{C}(\psi)$, the  class of subordinants $S_{f}(\psi):= \{g : g\prec f  \} $ is studied for the Bohr-Rogosinski phenomenon in general settings along with its applications to the standard classes of univalent starlike and convex functions.

\section{Bohr-Rogosinski phenomenon}\label{sec-1}
The following fundamental result is an extention of the Lemma~\ref{Bhowmik-lemma}:
\begin{lemma}\label{series-lem}
	let $f(z)=\sum_{n=0}^{\infty}a_n z^n$ and $g(z)=\sum_{k=0}^{\infty}b_k z^k$ be analytic in $\mathbb{D}$ and $g\prec f$, then 
	\begin{equation}\label{N-inequality}
	\sum_{k=N}^{\infty}|b_k|r^k \leq \sum_{n=N}^{\infty}|a_n|r^n
	\end{equation}
	for $|z|=r\leq \frac{1}{3}$ and $N\in \mathbb{N}$.
\end{lemma}
\begin{proof}
	Since $g\prec f$, we have $g(z)=f(\omega(z))$, where $\omega$ is a Schwarz function. For the case $\omega(z)=cz$, $|c|=1$, the function $g$ is a rotation of $f$ or $g=f$, and the inequality \eqref{N-inequality} easily holds. So consider the case: $\omega(z) \neq cz$, $|c|=1$. Now the coefficient $b_k$ of the function $g$	is given by: for any $k\geq N \in \mathbb{N}$
	\begin{equation*}
	b_k=\sum_{n=N}^{k}a_n {\beta_k}^{(n)},
	\end{equation*}
	where the $t$-th power of the analytic function $\omega$ is represented as $\omega^t(z)=\sum_{l\geq t}^{} {\beta_{l}}^{(t)} z^l$, $t\in \mathbb{N}$. Now we see that
	\begin{align*}
	\sum_{k=N}^{m}|b_k| r^k &= \sum_{k=N}^{m} \left|\sum_{n=N}^{n} a_n {\beta_k}^{(n)} \right| r^k \\
	&\leq \sum_{k=N}^{m} \sum_{n=N}^{n} |a_n| |{\beta_k}^{(n)}| r^k \\
	&= \sum_{n=N}^{m}|a_n| {M_m}^{(n)}(r),
	\end{align*}
	where ${M_m}^{(n)}(r)= \sum_{k=n}^{m}|{\beta_k}^{(n)}| r^k$ and $m\in \mathbb{N}$. Since $|\omega^n(z)/ z^n|<1$ for any $n\geq 1$, using Bohr's Theorem~\ref{BohrTheorem} we have
	\begin{equation*}
	\sum_{k=n}^{m} |{\beta_k}^{(n)}|r^{k-n} \leq \sum_{k=n}^{\infty} |{\beta_k}^{(n)}|r^{k-n} \leq 1, \quad r\leq \frac{1}{3},
	\end{equation*}
	that is, ${M_m}^{(n)}(r) \leq r^n$ holds for $r\leq 1/3$. Hence, for any $m\geq N\geq 1$ and $r\leq 1/3$
	\begin{equation*}
	\sum_{k=N}^{m}|b_k| r^k \leq \sum_{n=N}^{m} |a_n|r^n.
	\end{equation*}
	The result now follows by taking $m \rightarrow \infty$. \qed
\end{proof}	

\begin{proof}[Alternate proof of the Lemma~\ref{series-lem}]
	Since $g(z)=f(\omega(z))$, where $\omega$ is the Schawrz function, we have
	\begin{align*}
	M^{N}_r(g) &=M^{N}_r \left( \sum_{k=N}^{\infty}a_k (\omega(z))^k \right) \\
	& \leq \sum_{k=N}^{\infty} |a_k| \left( M_r(\omega(z))  \right)^k \\
	& \leq \sum_{k=N}^{\infty} |a_k| |z|^k
	\end{align*}
	for $|z|=r\leq 1/3$. \qed
\end{proof}	
\begin{remark}
	In Lemma~\ref{series-lem}, taking $N\rightarrow 1$ and the fact the $g(0)=f(0)$ we obtain Lemma~\ref{Bhowmik-lemma}.
\end{remark}

Moreover, the following results is obtained using the properties of the operator $M^{N}_r(f)$ and Lemma~\ref{series-lem}:
\begin{corollary}\label{quasi-serieslem}
	Let the analytic functions $f, g$ and $h$ satisfies $g(z)=h(z)f(\omega(z))$ in $\mathbb{D}$, where $\omega$ is the Schawrz function. Assume $|h(z)|\leq \tau$ for $|z|< \tau\leq1$. Then
	\begin{equation*}
	M^{N}_r(g) \leq \tau M^{N}_r(f), \quad 0\leq |z|=r\leq \frac{\tau}{3}.
	\end{equation*}
\end{corollary}
\begin{corollary}
	Let $\tau=1$ in Theorem~\ref{quasi-serieslem}. Then 
	\begin{equation*}
	M^{N}_r(g) \leq M^{N}_r(f), \quad 0\leq |z|=r\leq \frac{1}{3}.
	\end{equation*}
\end{corollary}	

\begin{lemma}\emph{(\cite{minda94})}\label{grth}
	Let $f\in \mathcal{S}^*(\psi)$ and $|z_0|=r<1$. Then $f(z)/z \prec f_0(z)/z$ and
	$$-f_0(-r)\leq|f(z_0)|\leq f_0(r).$$
	Equality holds for some $z_0\neq0$ if and only if $f$ is a rotation of $f_0$, where $zf_0(z)/f_0(z)=\psi(z)$ such that
	\begin{equation}\label{int-rep}
	f_0(z)=z\exp{\int_{0}^{z}\frac{\psi(t)-1}{t}dt}.
	\end{equation}
\end{lemma}

Our next results discuss Bohr-Rogosinski phenomenon for the classes $S_{f}(\psi)$ and $\mathcal{S}^*(\psi)$, respectively.
\begin{theorem}\label{BR-S}
	Let $r_{*}$ be the Koebe-radius for the class $\mathcal{S}^*(\psi),$  $f_0(z)$ be given by the equation~\eqref{int-rep} and $f(z)=z+\sum_{n=2}^{\infty}a_n z^n \in \mathcal{S}^*(\psi)$. Assume  $f_0(z)=z+\sum_{n=2}^{\infty}t_n z^n$ and $\hat{f}_0(r)=r+\sum_{n=2}^{\infty}|t_n|r^n$. If $g\in S_{f}(\psi)$. Then 
	\begin{equation}\label{BR-sf-Ineq}
	|g(z^m)| + \sum_{k=N}^{\infty}|b_k||z|^k \leq d(0, \partial{\Omega})
	\end{equation}
	holds for $|z|=r_b \leq \min \{ \frac{1}{3}, r_0 \}$, where $m, N\in \mathbb{N}$, $\Omega=f(\mathbb{D})$ and $r_0$ is the unique positive root of the equation:
	\begin{equation}
	\hat{f}_0 (r^m) + \hat{f}_0 (r) -p_{\hat{f}_0}(r)=r_{*},
	\end{equation}
	where 
	\begin{equation*}
	p_{\hat{f}_0}(r)=
	\left\{
	\begin{array}
	{lr}
	0, & N=1; \\
	r,   & N=2\\
	r+\sum_{n=2}^{N-1}|t_n|r^n, & N\geq3
	\end{array}
	\right.
	\end{equation*}
	The result is sharp when $r_b=r_0$ and $t_n>0$.
\end{theorem}
\begin{proof}
	Let $g(z)=\sum_{k=1}^{\infty}b_k z^k \prec f(z)$, where $f\in \mathcal{S}^*(\psi)$. Now by Lemma~\ref{series-lem}, for $r\leq 1/3$, we have
	\begin{equation*}
	\sum_{k=N}^{\infty}|b_k|r^k \leq \sum_{n=N}^{\infty}|a_n|r^n.
	\end{equation*}
	Again applying Lemma~\ref{series-lem} on $f(z)/z \prec f_0(z)/z$ (Lemma~\ref{grth}), we get that
	\begin{equation}\label{BR-S1}
	\sum_{k=N}^{\infty}|b_k|r^k \leq \sum_{n=N}^{\infty}|a_n|r^n  \leq \sum_{n=N}^{\infty}|t_n|r^n, \quad r\leq \frac{1}{3}.
	\end{equation}
	Now $g\prec f$ implies that $g(z)=f(\omega (z))$, which using the Lemma~\ref{grth} yields
	\begin{equation*}
	|g(z)|=|f(\omega (z))| \leq f_0(r)
	\end{equation*}
	for $|z|=r$, where $\omega$ is a Schwarz function. Moreover,
	\begin{equation}\label{BR-S2}
	|g(z^m)| \leq \hat{f}_0(r^m).
	\end{equation}
	Also, by letting $r$ tends to $1$ in Lemma~\ref{grth}, we obtain the Koebe-radius $r_{*}=-f_0(-1)$. Therefore, the open ball
	$\mathbb{B}(0,r_{*}) \subset f(\mathbb{D})$, which implies that for $|z|=1$
	\begin{equation}\label{BR-S3}
	r_{*}\leq d(0,\partial\Omega).
	\end{equation}
	Now using the equations \eqref{BR-S1}, \eqref{BR-S2} and \eqref{BR-S3}, we have
	\begin{align*}
	|g(z^m)| + \sum_{k=N}^{\infty}|b_k||z|^k &\leq \hat{f}_0(r^m)+ \sum_{n=N}^{\infty}|t_n|r^n\\
	& = \hat{f}_0(r^m)+ \hat{f}_0(r)-p_{\hat{f}_0}(r)\\
	& \leq r_* \\
	&\leq d(0, \partial \Omega) 
	\end{align*}
	holds whenever $|z|=r \leq \min \{ \frac{1}{3}, r_0 \}$, where $r_0$ is the smallest positive root of the equation:
	\begin{equation*}
	G(r):= \hat{f}_0(r^m)+ \hat{f}_0(r)-p_{\hat{f}_0}(r)-r_{*}=0.
	\end{equation*}
	Note that $G(0)<0$, and since $\hat{f}_0(1)\geq |f_0(1)|\geq r_{*}$, we see that
	\begin{align*}
	2\hat{f}_0(1) -\sum_{n=1}^{N-1}|t_n|-r_{*}
	=(\hat{f}_0(1)-\sum_{n=1}^{N-1}|t_n|)+ (\hat{f}_0(1)-r_{*})>0
	\end{align*}
	where $t_1=1$, which implies $G(1)>0$. 	Clearly, for $0\leq r\leq1$
	$$G'(r)=\hat{f}{'}_0(r^m)+ (\hat{f}{'}_0(r)-p{'}_{\hat{f}_0}(r))>0,$$
	which implies $G$ is a continuous increasing function in $[0,1]$. Thus $G(r)=0$ has a root in the interval $(0,1)$. The sharpness follows for the function $f_0$  as
	\begin{equation*}
	f_0({r_b}^m) + \sum_{n=N}^{\infty}t_n{r_b}^n=r^*= d(0, \partial \Omega)
	\end{equation*}
	when $r_b=r_0$ and $t_n>0$.\qed
\end{proof}
\begin{remark}
	Let $\psi(z)=(1+z)/(1-z)$, then Theorem~\ref{BR-S} reduces to \cite[Theorem~5]{Kayumov-2021}.
\end{remark}	
\begin{remark}
	Observe that if we take $m\rightarrow \infty$ and $N=1$, then Theorem~\ref{BR-S} reduces to \cite[Theorem~5.1]{gangania-bohr}.	
\end{remark}	

\begin{corollary}
	Let $r_{*}$ be the Koebe-radius for the class $\mathcal{S}^*(\psi),$  $f_0(z)$ be given by the equation~\eqref{int-rep}. Assume  $f_0(z)=z+\sum_{n=2}^{\infty}t_n z^n$ and $\hat{f}_0(r)=r+\sum_{n=2}^{\infty}|t_n|r^n$. If $f(z)=z+\sum_{n=2}^{\infty}a_n z^n \in \mathcal{S}^*(\psi)$. Then 
	\begin{equation}\label{BR-Spsi-Ineq}
	|f(z^m)| + \sum_{n=N}^{\infty}|a_n||z|^n \leq d(0, \partial{\Omega})
	\end{equation}
	holds for $|z|=r_b \leq \min \{ \frac{1}{3}, r_0 \}$, where $m, N\in \mathbb{N}$, $\Omega=f(\mathbb{D})$ and $r_0$ is the unique positive root of the equation:
	\begin{equation*}
	\hat{f}_0 (r^m) + \hat{f}_0 (r) -p_{\hat{f}_0}(r)=r_{*},
	\end{equation*}
	where $p_{\hat{f}_0}$ is as defined in Theorem~\ref{BR-S}. The radius is sharp for the function $f_0$ when $r_b=r_0$ and $t_n>0$.
\end{corollary}

\begin{corollary}\label{ravi-card}
	Let $\psi(z)=1+\dfrac{4}{3}z+\dfrac{2}{3}z^2$,  $f_0(r)=r \exp\left(\dfrac{4}{3}r+\dfrac{r^{2}}{3} \right)$ and $m=1$. If $g\in S_{f}(\psi)$. Then the inequality \eqref{BR-sf-Ineq}
	holds for $|z|=r\leq r_N$, where $N\in \mathbb{N}$ and $r_N (<1/3)$ is the unique positive root of the equation:
	\begin{equation*}
	2 r \exp\left(\frac{4}{3}r+\frac{r^{2}}{3} \right)-p_{f_0}(r) -\exp(-1)=0,
	\end{equation*}
	where $p_{f_0}=p_{\hat{f}_0}$ is as defined in Theorem~\ref{BR-S} with $|t_n|=t_n={f_0}^n(0) /n!$ . Moreover, if $f \in \mathcal{S}^*(\psi)$. Then the inequality \eqref{BR-Spsi-Ineq} also holds for $r\leq r_N$. The radius $r_N$ is sharp.
\end{corollary}	

\begin{remark}
	In Corollary~\ref{ravi-card}, we observe that the radius $r_N$ approaches $r_0=0.25588\cdots$ for large value of $N$, where $r_0$ is the unique positive root of
	\begin{equation*}
	r \exp\left(\frac{4}{3}r+\frac{r^{2}}{3} \right) -\exp(-1)=0.
	\end{equation*}
	Moreover, if $m\geq2$ then the inequalities \eqref{BR-sf-Ineq} and \eqref{BR-Spsi-Ineq} hold for $r\leq 1/3$.
\end{remark}

\begin{corollary}
	Let $\psi(z)=1+ze^z$ and $m=1$. If $g\in S_{f}(\psi)$. Then the inequality \eqref{BR-sf-Ineq}
	holds for $|z|=r\leq r_N =\{ r_0,1/3 \}$, where $N\in \mathbb{N}$ and $r_0$ is the unique positive root of the equation:
	\begin{equation*}
	2r\exp(e^r -1)-T(r) -\exp(e^{-1}-1)=0,
	\end{equation*}
	where
	\begin{equation*}
	T(r)=
	\left\{
	\begin{array}
	{lr}
	0, & N=1; \\
	r,   & N=2;\\
	\sum_{n=1}^{N-1}\frac{B_{n-1}}{(n-1)!} r^n, & N\geq3
	\end{array}
	\right.
	\end{equation*}
	and $B_n$ are the bell numbers such that $B_{n+1}=\sum_{k=0}^{n}{n \choose k} B_k$. Moreover, if $f \in \mathcal{S}^*(\psi)$. Then the inequality \eqref{BR-Spsi-Ineq} also holds for $r\leq r_N$. The radius $r_N <1/3$ is sharp for $N\leq3$.
\end{corollary}	

\begin{corollary}
	Let $\psi(z)=1+\frac{z}{k} \left( \frac{k+z}{k-z} \right)$ with $k=\sqrt{2}+1$. If $g\in S_{f}(\psi)$. Then the inequality \eqref{BR-sf-Ineq}
	holds for $|z|=r\leq r_b= \min\{1/3, r_0 \}$, where $N\in \mathbb{N}$ and $r_0 $ is the unique positive root of the equation:
	\begin{equation*}
	\frac{r^m}{e^{r^m}} \left( \frac{k}{k-r^m} \right)^{2k}+ \frac{r}{e^{r}} \left( \frac{k}{k-r} \right)^{2k} -p_{f_0}(r) -e\left( \frac{k}{k+1} \right)^{2k}=0,
	\end{equation*}
	where $p_{f_0}=p_{\hat{f}_0}$ is as defined in Theorem~\ref{BR-S} and $t_n=|t_n|$ are the Taylor coefficients of the function  $f_0(r)=\frac{r}{e^{r}} \left( \frac{k}{k-r} \right)^{2k}$. Moreover, if $f \in \mathcal{S}^*(\psi)$. Then the inequality \eqref{BR-Spsi-Ineq} also holds for $r\leq r_b$. The radius $r_b$ is sharp when $m=1$ and $N\leq4$.
\end{corollary}	

Since all the Taylor coefficients of the function $1+\sin{z}$ are not positive, $\hat{f}_0 \neq f_0$. So we consider the radius $r_N$ upto three decimal places only, which also reveals the connection of positive coefficients of $\psi$ to the sharp Bohr-Rogosinski radius.
\begin{corollary}\label{sin-BR-S}
	Let $\psi(z)=1+\sin{z}$ and $m=1$. If $g\in S_{f}(\psi)$. Then the inequality \eqref{BR-sf-Ineq}
	holds for $|z|=r\leq r_N$, where $N\in \mathbb{N}$ and $r_N (<1/3)$ is the unique positive root of the equation:
	\begin{equation*}
	2 r \exp(Si(r))-\exp(Si(-1))-p_{f_0}(r)=0,
	\end{equation*}
	where  $f_0(r)=r \exp(Si(r))$, where $Si(x)$ is the Sin Integral defined as:
	\begin{equation*}
	Si(x):=\int_{0}^{x}\frac{\sin(x)}{x}dx= \sum_{n=0}^{\infty}\frac{(-1)^n x^{2n+1}}{(2n+1)(2n+1)!}
	\end{equation*}
	Moreover, if $f \in \mathcal{S}^*(\psi)$. Then the inequality \eqref{BR-Spsi-Ineq} also holds for $r\leq r_N$. 
\end{corollary}	

\begin{remark}
	In Corollary~\ref{sin-BR-S}, the numerical computations reveal that the Bohr-Rogosinski radius $r_N \approx 0.290*\cdots <1/3$ for any $N>4$, where $*=6$ or $7$. Also $r_N<1/3$ for $N\leq4$. Moreover, as $N\rightarrow \infty$, the required radius $r_0\approx 0.290*\cdots$ is the unique positive root of
	\begin{equation*}
	r \exp(Si(r))-\exp(Si(-1))=0.
	\end{equation*}
\end{remark}

Next we discuss the Bohr-Rogosinski phenomenon for the celebrated Janowski class of univalent starlike functions. For this, we first need the following: for simplicity write $\mathcal{S}^*((1+Dz)/(1+Ez))\equiv \mathcal{S}[D,E] $, where $-1\leq E<D\leq1$.
\begin{lemma}\label{jan-coef}\cite[Theorem~3]{Aouf-1987}
	If $f(z)=z+\sum_{n=2}^{\infty}a_n z^n \in \mathcal{S}[D,E]$. Then for $n\geq2$, the following sharp bounds occur:
	\begin{equation*}
	|a_n| \leq \prod_{k=0}^{n-2}\frac{|E-D+Ek|}{k+1}.
	\end{equation*}	
\end{lemma}

\begin{corollary}\label{Jan-BR}
	Let $\psi(z)=({1+Dz})/({1+Ez})$, $-1\leq E< D\leq1$. If $f(z)=z+\sum_{n=2}^{\infty}a_n z^n \in \mathcal{S}^*(\psi)$. Then 
	\begin{equation}\label{J-alpha}
	|f(z^m)| + \sum_{n=N}^{\infty}|a_n||z|^n \leq d(0, \partial{\Omega})
	\end{equation}
	holds for $|z|=r\leq r_0$, where $m, N\in \mathbb{N}$, $\Omega=f(\mathbb{D})$ and $r_0$ is the unique positive root of the equations:
	\begin{equation*}
	r^m(1+Er^m)^{\frac{D-E}{E}} +A(r)+ \sum_{n=N}^{\infty} \prod_{k=0}^{n-2}\frac{|E-D+Ek|}{k+1} r^n -(1-E)^{\frac{D-E}{E}}=0, \quad\text{if}\;  E\neq0,
	\end{equation*}
	where $A(r)=r$ for $N=1$ and $0$ otherwise, and
	\begin{equation*}
	r^m e^{Dr^m} + re^{Dr}-J(r) -e^{-D}=0, \quad\text{if}\;  E=0,
	\end{equation*}
	where
	\begin{equation}\label{Jan-E0}
	J(r)=
	\left\{
	\begin{array}
	{lr}
	0, & N=1; \\
	r, & N=2;\\
	\sum_{n=2}^{N-1} \prod_{k=0}^{n-2}\frac{D}{k+1} r^n,   & N\geq3.
	\end{array}
	\right.
	\end{equation}
	The radius $r_0$ is sharp.
\end{corollary}
\begin{proof}
	Let us consider the function $f_0$ such that $zf'_{0}(z)/ f_{0}(z)= ({1+Dz})/({1+Ez})$, which is given by
	\begin{equation}\label{Janow-extFunct}
	f_0(z)= 
	\left\{
	\begin{array}
	{lr}
	z(1+E z)^{\frac{D-E}{E}}, & E\neq0; \\
	ze^{Dz},   & E=0.
	\end{array}
	\right.
	\end{equation}
	Now using the Lemma~\ref{grth} and Lemma~\ref{jan-coef}, we have
	\begin{equation*}
	|f(z^m)|\leq f_0(r^m), \quad r_{*}= -f_0(-1)
	\end{equation*}
	and 
	\begin{equation*}
	\sum_{n=N}^{\infty}|a_n||z|^n \leq \sum_{n=N}^{\infty} \prod_{k=0}^{n-2}\frac{|E-D+Ek|}{k+1} r^n, \quad N\geq2.
	\end{equation*}
	Now proceeding as in Theorem~\ref{BR-S}, for $r_0$ as defined in the statement, the result follows. To prove the sharpness of the radius $r_0$, we see that at $|z|=r=r_0$ and $f=f_0$ given in \eqref{Janow-extFunct}:
	\begin{align*}
	&|f(z^m)| + \sum_{n=N}^{\infty}|a_n||z|^n \\
	&=  
	\left\{
	\begin{array}
	{lr}
	(r_0)^m(1+E(r_0)^m)^{\frac{D-E}{E}} +A(r_0)+ \sum_{n=N}^{\infty} \prod_{k=0}^{n-2}\frac{|E-D+Ek|}{k+1} (r_0)^n, & E\neq0; \\
	(r_0)^m e^{D(r_0)^m} + (r_0)e^{Dr_0}-J(r_0),   & E=0.
	\end{array}
	\right. \\
	&= 
	\left\{
	\begin{array}
	{lr}
	(1-E)^{\frac{D-E}{E}}, & E\neq0; \\
	e^{-D},   & E=0.
	\end{array}
	\right.
	\\
	&=-f_0(-1)\\
	&= d(0, \partial{\Omega}),
	\end{align*}	
	where $J(r)$ is as defined in \eqref{Jan-E0}, and $A(r)=r$ for $N=1$ and $0$ otherwise for the case $E\neq0$. \qed
\end{proof}	
\begin{remark}
	Taking $m\rightarrow \infty$ and $N=1$ in Corollary~\ref{Jan-BR}, we obtain the Bohr radius for the class $\mathcal{S}[D,E]$, which covers many classical cases.
\end{remark}	

In Corollary~\ref{Jan-BR}, putting $D=1-2\alpha$ and $E=-1$, where $0\leq\alpha<1$, we get the result for the class of univalent starlike functions of order $\alpha$, that is, $\mathcal{S}^{*}(\alpha)$:
\begin{corollary}\label{starlikealpha}
	If $f(z)=z+\sum_{n=2}^{\infty}a_n z^n \in \mathcal{S}^{*}(\alpha)$. Then the inequality \eqref{J-alpha} holds for $|z|=r\leq r_0$, where $m, N\in \mathbb{N}$, $\Omega=f(\mathbb{D})$ and $r_0$ is the smallest positive root of the equations:
	\begin{equation*}
	\frac{r^m}{(1-r^m)^{2(1-\alpha)}} +A(r)+ \sum_{n=N}^{\infty} \prod_{k=0}^{n-2}\frac{k+2(1-\alpha)}{k+1} r^n -\frac{1}{4^{1-\alpha}}=0,
	\end{equation*}
	where $A(r)=r$ for $N=1$ and $0$ otherwise. The radius $r_0$ is sharp.
\end{corollary}	

Putting $\alpha=0$ in Corollary~\ref{starlikealpha}, we get the following:
\begin{corollary}
	If $f(z)=z+\sum_{n=2}^{\infty}a_n z^n \in \mathcal{S}^{*}$. Then the inequality \eqref{J-alpha} holds for $|z|=r\leq r_0$, where $m, N\in \mathbb{N}$, $\Omega=f(\mathbb{D})$ and $r_0$ is the smallest positive root of the equations:
	\begin{equation*}
	4r^m- (1-r^m)^2+ 4r^N (N(1-r)+r) \left( \frac{1-r^m}{1-r} \right)^2 =0.
	\end{equation*}
	The radius $r_0$ is sharp.
\end{corollary}	

To proceed further, we need to recall the following fundamental result:
\begin{lemma}\label{C-grth}\cite{minda94}
	Let $f\in \mathcal{C}(\psi)$. Then $zf''(z)/f'(z) \prec zl_{0}''(z)/l_{0}'(z)$ and $f'(z)\prec l_{0}'(z)$. Also, for $|z|=r$ we have 
	$$-l_{0}(-r)\leq |f(z)| \leq l_{0}(r),$$
	where 
	\begin{equation}\label{int-rep-C}
	zl_{0}''(z)/l_{0}'(z)=\psi(z).
	\end{equation}
\end{lemma}	

Now we discuss the results for the convex analogue $\mathcal{C}(\psi)$ of $\mathcal{S}^*(\psi)$.
\begin{theorem}\label{BR-Convexcase}
	Let $r_{*}$ be the Koebe-radius for the class $\mathcal{C}(\psi),$  $l_0(z)$ be given by the equation~\eqref{int-rep-C} and $f(z)=z+\sum_{n=2}^{\infty}a_n z^n \in \mathcal{C}(\psi)$. Assume  $l_0(z)=z+\sum_{n=2}^{\infty}l_n z^n$ and $\hat{l}_0(r)=r+\sum_{n=2}^{\infty}|l_n|r^n$. If $g\in S_{f}(\psi)$. Then 
	\begin{equation}\label{Sf-convex}
	|g(z^m)| + \sum_{k=N}^{\infty}|b_k||z|^k \leq d(0, \partial{\Omega})
	\end{equation}
	holds for $|z|=r_b \leq \min \{ \frac{1}{3}, r_0 \}$, where $m, N\in \mathbb{N}$, $\Omega=f(\mathbb{D})$ and $r_0$ is the unique positive root of the equation:
	\begin{equation*}
	\hat{l}_0 (r^m) + \hat{l}_0 (r) -p_{\hat{l}_0}(r)=r_{*},
	\end{equation*}
	where 
	\begin{equation*}
	p_{\hat{l}_0}(r)=
	\left\{
	\begin{array}
	{lr}
	0, & N=1; \\
	r,   & N=2;\\
	r+\sum_{n=2}^{N-1}|l_n|r^n, & N\geq3.
	\end{array}
	\right.
	\end{equation*}
	The result is sharp when $r_b=r_0$ and $l_n>0$.
\end{theorem}
\begin{proof}
	Let $g(z)=\sum_{k=1}^{\infty}b_k z^k \prec f(z)$, where $f\in \mathcal{C}(\psi)$. From the Alexender relation, it is known that $f\in \mathcal{C}(\psi)$ if and only if $$zf'(z)=\tilde{g}(z),\quad \text{or equivalently} \quad f(z)=\int_{0}^{z}\frac{\tilde{g}(t)}{t}dt$$
	for some $\tilde{g}\in \mathcal{S}^*(\psi)$.
	Now by Lemma~\ref{series-lem}, for $r\leq 1/3$, we have
	\begin{equation}\label{BR-C0}
	\sum_{k=N}^{\infty}|b_k|r^k \leq \sum_{n=N}^{\infty}|a_n|r^n= \sum_{n=N}^{\infty}\frac{|\tilde{b}_n|}{n}r^n,
	\end{equation}
	where $\tilde{b}_n$ are the Taylor coefficients of $\tilde{g}$.
	Again applying Lemma~\ref{series-lem} on $f'(z) \prec l'_0(z)$ (Lemma~\ref{C-grth}), we get that
	\begin{equation}\label{BR-C1}
	M_{\tilde{g}}(r)-p_{\tilde{g}}(r) \leq M_{h}(r)-p_h(r) , \quad r\leq \frac{1}{3},
	\end{equation}
	where $M{g}(x):=\sum_{k=1}^{\infty}|b_k|x^k$, and $h$ is given by the relation $zl'_0(z)=h(z)$.
	Now using the equations \eqref{BR-C0} and \eqref{BR-C1}, we have for $r\leq 1/3$
	\begin{align}\label{second-sum}
	\sum_{k=N}^{\infty}|b_k||z|^k &\leq \sum_{n=N}^{\infty}\frac{|\tilde{b}_n|}{n}r^n \nonumber\\
	&= \int_{0}^{r}\frac{M_{\tilde{g}(t)}-p_{\tilde{g}(t)}}{t}dt \nonumber \\
	&\leq \int_{0}^{r}\frac{M_{h}(t)-p_{h}(t)}{t}dt
	= \sum_{n=N}^{\infty}|l_{0}|r^n  \nonumber\\
	&= \hat{l}_0 (r) -p_{\hat{l}_0}(r).
	\end{align}
	Now $g\prec f$ implies that $g(z)=f(\omega (z))$, which using the Lemma~\ref{C-grth} yields
	\begin{equation*}
	|g(z)|=|f(\omega (z))| \leq l_0(r)
	\end{equation*}
	for $|z|=r$, where $\omega$ is a Schwarz function. Moreover,
	\begin{equation}\label{BR-C2}
	|g(z^m)| \leq \hat{l}_0(r^m).
	\end{equation}
	Also, by letting $r$ tends to $1$ in Lemma~\ref{C-grth}, we obtain the Koebe-radius $r_{*}=-l_0(-1)$. Therefore, the open ball
	$\mathbb{B}(0,r_{*}) \subset f(\mathbb{D})$, which implies that for $|z|=1$
	\begin{equation}\label{BR-C3}
	r_{*}\leq d(0,\partial\Omega).
	\end{equation}
	Hence, using the inequalities \eqref{second-sum}, \eqref{BR-C2} and \eqref{BR-C3}, we have
	\begin{align*}
	|g(z^m)|+ \sum_{k=N}^{\infty}|b_k||z|^k &\leq \hat{l}_0 (r^m)+ \hat{l}_0 (r) -p_{\hat{l}_0}(r)\\
	& \leq r_{*} \\
	&\leq d(0,\partial \Omega)
	\end{align*}
	holds whenever $|z|=r \leq \min \{ \frac{1}{3}, r_0 \}$, where $r_0$ is the smallest positive root of the equation:
	\begin{equation*}
	H(r):= \hat{l}_0(r^m)+ \hat{l}_0(r)-p_{\hat{l}_0}(r)-r_{*}=0.
	\end{equation*}
	Clearly,  $H$ is continuous and $H'(r)>0$ for $0\leq r\leq1$. Note that $H(0)<0$, and since $\hat{l}_0(1)\geq |l_0(1)|\geq r_{*}$, we see that
	\begin{align*}
	2\hat{l}_0(1) -\sum_{n=1}^{N-1}|l_n|-r_{*}
	=(\hat{l}_0(1)-\sum_{n=1}^{N-1}|l_n|)+ (\hat{l}_0(1)-r_{*})
	>0,
	\end{align*}
	which implies $H(1)>0$. Thus $H(r)=0$ has a root in the interval $(0,1)$. The sharpness follows for the function $l_0$ as
	\begin{equation*}
	l_0(r_{b}^m)+ \sum_{n=N}^{\infty}l_n r_{b}^n = r_{*} = d(0,\partial \Omega)
	\end{equation*}
	when $r_b=r_0$ and $l_n>0$. \qed
\end{proof}	
\begin{remark}
	Let $\psi(z)=(1+z)/(1-z)$, then Theorem~\ref{BR-Convexcase} reduces to \cite[Theorem~6]{Kayumov-2021}.
\end{remark}	

The following result is explicitly for the class $\mathcal{C}(\psi)$.
\begin{corollary}\label{convexclass-BR}
	Let $r_{*}$ be the Koebe-radius for the class $\mathcal{C}(\psi),$  $l_0(z)$ be given by the equation~\eqref{int-rep-C}. Assume  $l_0(z)=z+\sum_{n=2}^{\infty}l_n z^n$ and $\hat{l}_0(r)=r+\sum_{n=2}^{\infty}|l_n|r^n$. If $f(z)=z+\sum_{n=2}^{\infty}a_n z^n \in \mathcal{C}(\psi)$. Then 
	\begin{equation}\label{BR-convexCase}
	|f(z^m)| + \sum_{n=N}^{\infty}|a_n||z|^n \leq d(0, \partial{\Omega})
	\end{equation}
	holds for $|z|=r_b \leq \min \{ \frac{1}{3}, r_0 \}$, where $m, N\in \mathbb{N}$, $\Omega=f(\mathbb{D})$ and $r_0$ is the unique positive root of the equation:
	\begin{equation*}
	\hat{l}_0 (r^m) + \hat{l}_0 (r) -p_{\hat{l}_0}(r)=r_{*},
	\end{equation*}
	where $p_{\hat{l}_0}$ is as defined in Theorem~\ref{BR-Convexcase}. The radius is sharp for the function $l_0$ when $r_b=r_0$ and $l_n>0$.
\end{corollary}	
\begin{remark}
	The special case of taking $m\rightarrow \infty$ and $N=1$ in Theorem~\ref{Sf-convex} and Corollary~\ref{convexclass-BR} establish the Bohr phenonmenon for the classes $S_{f}(\psi)$ and $\mathcal{C}(\psi)$, respectively.
\end{remark}	

After some little computations when $\psi(z)=(1+z)/(1-z)$, the Corollary~\ref{convexclass-BR} yields:
\begin{corollary}
	If $f(z)=z+\sum_{n=2}^{\infty}a_n z^n \in \mathcal{C}$. Then the inequality \eqref{BR-convexCase} holds for $|z|=r\leq r_0$, where $m, N\in \mathbb{N}$, $\Omega=f(\mathbb{D})$ and $r_0$ is the unique positive root of the equations:
	\begin{equation*}
	3r^m- 1+ 2r^N \left( \frac{1-r^m}{1-r} \right) =0.
	\end{equation*}
	The radius $r_0$ is sharp.
\end{corollary}

\begin{corollary}
	Let $\psi(z)=1+ze^z$ and $m=1$. If $g\in S_{f}(\psi)$. Then the inequality \eqref{Sf-convex}
	holds for $|z|=r\leq r_N$, where $N\in \mathbb{N}$ and $r_N (<1/3)$ is the unique positive root of the equation:
	\begin{equation*}
	2r(1+re^r) \exp(e^r -1) -H(r) -(1-e^{-1}) e^{e^{-1}-1}=0,
	\end{equation*}
	where
	\begin{equation*}
	H(r)=
	\left\{
	\begin{array}
	{lr}
	0, & N=1; \\
	r, & N=2;\\
	\sum_{n=0}^{N-1}{(n+1)B_n \choose n!}r^{n+1} ,   & N\geq3.
	\end{array}
	\right.
	\end{equation*}
	and $B_n$ are the bell numbers such that $B_{n+1}=\sum_{k=0}^{n}{n \choose k} B_k$.
	Moreover, if $f \in C(\psi)$. Then the inequality \eqref{BR-convexCase} also holds for $r\leq r_N$. The radius $r_N$ is sharp.
\end{corollary}

\section*{Conflict of interest}
	The authors declare that they have no conflict of interest.

\end{document}